%% file: tetrus6.tex
\documentclass{amsart}
\usepackage{amsmath,amssymb,amsthm}
\usepackage[all]{xy}
\usepackage[pdftex]{graphicx}
\usepackage{color}

\newcommand{\co}{\colon \thinspace}
\newcommand{\onto}{\twoheadrightarrow}

\newtheorem{thm}{Theorem}[section]
\newtheorem{cor}[thm]{Corollary}
\newtheorem*{thm1}{Theorem}
\newtheorem{prop}{Proposition}[section]
\newtheorem{lemma}[prop]{Lemma}
\newtheorem{propo}[thm]{Proposition}

\theoremstyle{definition}
\newtheorem*{dfn}{Definition}

\newtheorem*{fact}{Fact}

\theoremstyle{remark}
\newtheorem*{remark}{Remark}

\begin{document}

\title{On the doubled tetrus}
\author{Jason DeBlois}
\address{Department of MSCS, University of Illinois at Chicago}
\email{jdeblois@math.uic.edu}
\thanks{Partially supported by NSF grant DMS-0703749}
\begin{abstract}  The ``tetrus'' is a member of a family of hyperbolic $3$-manifolds with totally geodesic boundary, described by Paoluzzi-Zimmerman, which also contains W.P. Thurston's ``tripus.''  Each member of this family has a branched cover to $B^3$, branched over a certain tangle $T$.  This map on the tripus has degree three, and on the tetrus degree four.  We describe a cover of the double of the tetrus, itself a double across a closed surface, which fibers over the circle.
\end{abstract}
\maketitle

This paper describes some features of a fibered cover of a certain $3$-manifold, related to a family of compact hyperbolic 3-manifolds $M_{n,k}$ with totally geodesic boundary defined by Paoluzzi-Zimmermann \cite{ZP}.  A well known member of this family is $M_{3,1}$, Thurston's ``tripus'' \cite[Ch. 3]{Th1}.  Here we consider the ``tetrus" $M_{4,1}$ (thanks to Richard Kent for naming suggestions).  

If $M$ is an oriented manifold with boundary, let $\overline{M}$ be a copy of $M$ with orientation reversed, and define the \textit{double} of $M$ to be $DM = M \cup_{\partial} \overline{M},$  where the gluing isometry $\partial M \to \partial \overline{M}$ is induced by the identity map.  

\begin{thm} \label{fibered}  There is a cover $p \co D\widetilde{M} \to DM_{4,1}$ of degree $6$, where $D\widetilde{M}$ is a double across the closed surface $p^{-1}(\partial M_{4,1})$ of genus $13$, and $D\widetilde{M}$ fibers over $S^1$ with fiber $\widetilde{F}$, a closed surface of genus $19$.  
\end{thm}

The manifold $D\widetilde{M}$ above is the first hyperbolic $3$-manifold which we know to be both fibered and a double across a closed surface.  Non-hyperbolic fibered doubles are easily constructed, for instance by doubling the exterior of a fibered knot across its boundary torus, but producing a fibered hyperbolic double is more subtle problem.  In such a manifold the doubling surface --- necessarily with genus at least 2 --- does not itself admit a fibering and must thus have points of tangency with any fibering, which in particular cannot be invariant under the doubling involution.

The strategy of proof for Theorem \ref{fibered} is motivated by the fact that the doubled tetrus branched covers $S^3$, branched over the link $L$ of Figure \ref{doubledtangle}, which is the Montesinos link $L(1/3,1/2,-1/2,-1/3)$.  Thurston observed that given a branched cover $M_n \to M$, branched over a link $L$, virtual fiberings of $M$ transverse to the preimage of $L$ may be pulled back to virtual fiberings of $M_n$.   This is also used in \cite{AR}, which is where we encountered it, and \cite{ABZ}.  We record a  version of this observation as Proposition \ref{technical_branched}, and apply it here with Proposition \ref{transverse} to prove Theorem \ref{fibered}.  

When $M'$ is homeomorphic to the Cartesian product of a closed surface $F$ with $S^1$, and $L$ is a link contained in a disjoint union of fibers, Proposition \ref{transverse} produces new fiberings of $M'$ transverse to $L$ under certain circumstances.  We state and prove Proposition \ref{transverse} in Section \ref{sec:Thurston}, along with Proposition \ref{technical_branched}.  In Section \ref{sec:tetrus} we describe the manifolds $M_{n,k}$ constructed in \cite{ZP}, and their branched covers to the $3$-ball.  The branched cover $DM_{4,1} \to S^3$, which results from doubling, factors through a double branched cover $DM_2 \rightarrow S^3$, branched over $L$.  Such a manifold is well known to have the structure of a Seifert fibered space (\cite{Mont}, cf. \cite{BuZi}).  In Section \ref{sec:fibering}, we describe the Seifert fibered structure on $DM_2$ and construct a cover $p' \co DM' \rightarrow DM_2$ which satisfies the hypotheses of Proposition \ref{transverse}.  An application of Proposition \ref{technical_branched} completes the proof of Theorem \ref{fibered}.

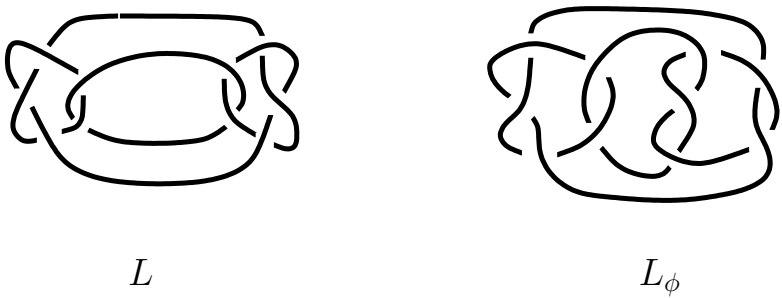
\begin{figure}
\begin{center}
\input{doubles.pdf_t.txt}
\end{center}
\caption{$DM_{n,k}$ and $D_{\phi}M_{n,k}$ $n$--fold cover $S^3$, branched over $L$ and $L_{\phi}$.}
\label{doubledtangle}
\end{figure}

Steven Boyer and Xingru Zhang previously obtained results about virtual fiberings of Montesinos links and their branched covers (see \cite{ABZ}) using a similar strategy.  These imply that the doubled tetrus is virtually fibered; in particular, \cite[Theorem 1.7]{ABZ} applies more generally than our independently discovered Proposition \ref{transverse}.  An advantage of Proposition \ref{transverse}, when it applies, is that it produces an explicit description of a fiber surface, which we use in Theorem \ref{fibered} to obtain the extra information about the genus of $\widetilde{F}$.

It is well known that $M_{n,k}$ has a minimal--genus Heegaard splitting of genus $n$, obtained by attaching a single one--handle to $\partial M_{n,k}$.  (Ushijima classified such splittings in \cite[Theorem 2.8]{Ush}.)  Thus the tetrus $M_{4,1}$ has a Heegaard splitting with genus 4, yielding an amalgamated Heegaard splitting of $DM_{4,1}$ with genus 5.  The preimage in $D\widetilde{M}$ gives the following corollary of Theorem \ref{fibered}.

\begin{cor}  $D\widetilde{M}$ has a weakly reducible Heegaard splitting of genus $25$ associated to $p^{-1}(\partial M_{4,1})$, and one of genus $39$ associated to $\widetilde{F}$.  \end{cor}

It would be interesting to know the minimal genus of a fiber surface for $D\widetilde{M}$, for the above discussion shows that if this is greater than twelve, the minimal genus Heegaard splitting of $D\widetilde{M}$ is not associated to a fibering.  Such examples are nongeneric, according to work of Souto \cite[Theorem 6.2]{Souto} and Biringer \cite{Biri}.  See also \cite{BDT} for a survey of results about degeneration.

The \textit{twisted double} $D_{\phi}M_{n,k}$ in the proposition is obtained by gluing $M_{n,k}$ to its mirror image via an isometry $\phi$ of the boundary, the lift to $M_{n,k}$ of the mutation producing the link $L_{\phi}$ of Figure \ref{doubledtangle}.  Our original motivation for proving Theorem \ref{fibered} was Proposition \ref{arithmetic} below, which describes arithmeticity of the doubles and twisted doubles for $n=3$ and $4$.

\begin{propo}  The doubles $DM_{3,k}$ ($k=0$ or 1) and $DM_{4,k}$ ($k=1$ or 3) are nonarithmetic.  On the other hand, $D_{\phi}M_{3,k}$ and $D_{\phi}M_{4,k}$ are arithmetic.  \label{arithmetic} \end{propo}

This may be verified using Snappea and Snap, together with the descriptions of $DM_{n,k}$ and $D_{\phi}(M_{n,k})$ as $n$--fold covers of $L$ and $L_{\phi}$ above.  For each $n$, $DM_{n,k}$ and $D_{\phi}M_{n,k}$ contain a totally geodesic surface identical to the totally geodesic boundary of $M_{n,k}$.  In the cases $n=3$ and $4$, it follows from arithmeticity of the twisted doubles that this surface is arithmetic.  (This can also be discerned directly from a polyhedral decomposition.)  Using Proposition 4.1 of \cite{LLR}, one obtains the following.

\begin{cor}  The doubled tetrus has a nested, cofinal family of regular covers with respect to which it has property $\tau$.  \label{LLR} \end{cor}

Work of Abert-Nikolov (\cite{AN}) concerning rank and Heegaard genus has drawn attention to virtually fibered manifolds which satisfy the conclusion of Corollary \ref{LLR}.  The doubled tetrus is the first closed nonarithmetic hyperbolic $3$-manifold that we knew to have both of these properties, although work of Agol \cite{Agol} shows that some nonarithmetic right--angled reflection groups, for instance in the L\"obell polyhedron $L(7)$, have them both as well.


\section*{Acknowledgments}

Thanks to Steve Boyer for explaining to me the results of \cite{ABZ}, and to Cameron Gordon and Ian Agol for helpful conversations.  Special thanks to Alan Reid for pointing me toward this problem and his continued help and guidance.


\section{Virtually fibering branched covers} \label{sec:Thurston}

We will be concerned in this paper with branched coverings.  The \textit{standard} $k$-fold cyclic branched covering of the disk $D^2$ to itself is the quotient map which identifies each point $z \in D^2 \subset \mathbb{C}$ with points of the form $ze^{2\pi i \frac{j}{k}}$, $0 \leq j < k$.  For a $3$-manifold $M$, an $n$-fold branched covering $q \co M_n \to M$, branched over a link $L \subset M$, is characterized by the property that $L$ has a closed regular neighborhood $\mathcal{N}(L)$, with \textit{exterior} $\mathcal{E}(L) = \overline{M - \mathcal{N}(L)}$, such that \begin{enumerate}
  \item  On $\mathcal{E}_n \doteq q^{-1}(\mathcal{E}(L))$, $q$ restricts to a genuine $n$-fold covering map.
  \item  Each component $V$ of $q^{-1}(\mathcal{N}(L))$ has a homeomorphism to $D^2 \times S^1$ so that $q|_V$ is the product of the standard $k$-fold branched cover with the identity map $S^1 \to S^1$, for some $k > 1$ dividing $n$.
\end{enumerate}

\begin{remark}  It might be more accurate to allow the map in condition 2 to be the product of the $k$-fold branched covering of $D^2$ with a nontrivial covering of $S^1$ to itself.  Since we will not encounter examples with this property here, we restrict our attention to the simpler setting.  \end{remark}

Now suppose $p' \co M' \rightarrow M$ is a genuine $g$-fold covering space, and let $\mathcal{E}' = (p')^{-1}(\mathcal{E}(L)) \subset M'$ be the associated cover of $\mathcal{E}(L)$.  The group $\pi(L) \doteq \pi_1(\mathcal{E}(L))$ has subgroups $\Gamma_n$, and $\Gamma'$, corresponding to the covers of $\mathcal{E}(L)$ by $\mathcal{E}_n$ and $\mathcal{E}'$, respectively.  Below we record an elementary observation about $\Gamma'$.

\begin{fact} Let $V \cong D^2 \times S^1$ be a component of $\mathcal{N}(L)$ and take $\mu = \partial D^2 \times \{y\}\subset \mathcal{E}(L)$ for some $y \in S^1$.  If $[\mu]$ represents the homotopy class of $\mu$ in $\pi(L)$, then $[\mu] \in \Gamma'$.  
\end{fact} 


This holds because $p'|_{\mathcal{E}'}$ extends to a covering map on $M'$; hence each component of the preimage of $\mu$ bounds a lift to $M'$ of the inclusion map $D^2 \times \{y\} \hookrightarrow M$.  We will call \textit{meridians} curves of the form $\partial D^2 \times \{y\} \subset \partial \mathcal{E}(L)$.  Take $\widetilde{\Gamma} = \Gamma_n \cap \Gamma'$, and let $\tilde{p} \co \widetilde{\mathcal{E}} \to \mathcal{E}$ be the associated covering space, factoring through the restriction of $q$ to $\mathcal{E}_n$ and of $p'$ to $\mathcal{E}'$.  We will say that $\widetilde{\mathcal{E}}$ \textit{completes the diamond}.  

\begin{lemma} \label{here's yer cover}  Let $p' \co M' \to M$ be a $g$-fold cover and $q_n \co M_n \rightarrow M$ an $n$-fold branched cover, branched over a link $L \subset M$.  Let $\mathcal{E}(L)$ be the exterior of $L$ and $\mathcal{E}'$, $\mathcal{E}_n$ its covers associated to $M'$ and $M_n$, respectively.  If $\tilde{p} \co \widetilde{\mathcal{E}} \to \mathcal{E}(L)$ completes the diamond, then $\tilde{p}$ extends to a map $\tilde{p} \co \widetilde{M} \rightarrow M$ such that $\tilde{p} = p' \circ q = q_n \circ p$ for a branched cover $q \co \widetilde{M} \rightarrow M'$ and a cover $p \co \widetilde{M} \to M_n$.
\end{lemma}

\begin{proof}  Since $\widetilde{\mathcal{E}}$ corresponds to the subgroup $\widetilde{\Gamma} = \Gamma_n \cap \Gamma' < \pi(L)$, there are coverings $p \co \widetilde{\mathcal{E}} \to \mathcal{E}_n$ and $q \co \widetilde{\mathcal{E}} \to \mathcal{E}'$ such that $\tilde{p} = q_n \circ p = p' \circ q$.

Let $V \cong D^2 \times S^1$ be a component of $\mathcal{N}(L)$, let $\lambda = \{x\}\times S^1$ for some $x \in D^2$, and let $\mu =\partial D^2 \times \{y\} \in \mathcal{E}(L)$ for some $y \in S^1$.  Taking $[\mu] \in \pi(L)$ to represent the homotopy class of $\mu$, we have $[\mu] \in \Gamma'$, so $[\mu]^k \in \widetilde{\Gamma}$ if and only if $[\mu]^k \in \Gamma_n$.  On the other hand, condition 2 in the characterization of branched coverings implies that each component of the preimage of $\lambda$ in $\mathcal{E}_n$ maps homeomorphically to $\lambda$ under $q_n$.  Thus $[\lambda] \in \Gamma_n$, so $[\lambda]^j \in \widetilde{\Gamma}$ if and only if $[\lambda]^j \in \Gamma'$.

Let $T$ be a component of $\partial \widetilde{\mathcal{E}}$ which maps under $\tilde{p}$ to $\partial V$, and let $\tilde{\mu}$ and $\tilde{\lambda}$ be components on $T$ of the preimages of $\mu$ and $\lambda$, respectively.  Then by the above, the map $p$ restricts to a homeomorphism of $\tilde{\mu}$ onto its image, and $\mathrm{deg}\, p|_{\tilde{\lambda}} = \mathrm{deg}\, p|_T$. On the other hand, $q$  has the same degree on $T$ as it does on $\tilde{\mu}$, and it restricts on $\lambda$ to a homeomorphism.

Let $\widetilde{\mathcal{E}}(\tilde\mu)$ be the manifold obtained from $\widetilde{\mathcal{E}}$ by Dehn filling along $\tilde\mu$.  That is $\widetilde{\mathcal{E}}(\tilde\mu)$ is the quotient of $\widetilde{\mathcal{E}} \sqcup \widetilde{V}$, where $\widetilde{V} \cong D^2 \times S^1$, by a homeomorphism $\partial D^2 \times S^1 \to T$ taking $\partial D^2 \times \{y\}$ to $\tilde{\mu}$ for some $y \in S^1$.  We may assume without loss of generality that this homeomorphism takes some $\{x\} \times S^1$ to $\tilde{\lambda}$.  Then $p$ extends to a covering map $\widetilde{\mathcal{E}}(\tilde{\mu}) \to \mathcal{E}_n(p(\tilde\mu))$, sending $\widetilde{V}$ to its image by the product of the identity map with the $(\mathrm{deg}\, p|_T)$-fold cover of $S^1$ to itself.   Likewise, $q$ extends to a branched cover $\widetilde{\mathcal{E}}(\tilde{\mu}) \to \mathcal{E}'(q(\tilde\mu))$.

We now take $\widetilde{M}$ to be the manifold obtained from $\widetilde{\mathcal{E}}$ by Dehn filling each boundary component along the preimage of the corresponding meridian.  It follows from the paragraph above that $p$ extends to the manifold obtained from $\mathcal{E}_n$ by filling boundary components along preimages of meridians.  By definition, this is $M_n$.  Similarly, $q$ extends to a branched covering from $\widetilde{M}$ to $M'$, establishing the lemma.  \end{proof}

Thurston used a version of Proposition \ref{technical_branched} below to show the reflection orbifold in a right--angled dodecahedron is virtually fibered (cf. \cite{Sul}); this fact is also used by Boyer-Zhang \cite[Cor 1.4]{ABZ}.  Our version explicitly describes a fibered cover.

\begin{prop} \label{technical_branched} Suppose $p' \co M' \to M$ is a $g$-fold cover and $q_n \co M_n \to M$ an $n$-fold branched cover, branched over a link $L \subset M$.  If $M'$ fibers over $S^1$ with fibers transverse to $(p')^{-1}(L)$, then the manifold $\widetilde{M}$ supplied by Lemma \ref{here's yer cover} fibers over $S^1$ in such a way that $q \co \widetilde{M} \to M'$ is fiber-preserving.  \end{prop}

\begin{proof}  Since $p^{-1}(L)$ is transverse to the fibering of $M$, for each component $V$ of $\mathcal{N}(L)$, a homeomorphism to $D^2 \times S^1$ may be chosen so that after an ambient isotopy of the fibering, for any component $\widetilde{V}$ of $(p')^{-1}(V)$, each fiber intersects $\widetilde{V}$ in a collection of disjoint disks of the form $D^2 \times \{y\}$ for $y \in S^1$.  Then $\mathcal{E}'$ inherits a fibering from $M'$ with the property that each fiber intersects the boundary in a collection of meridians.

By definition, each curve on $\partial \widetilde{\mathcal{E}}$ which bounds a disk in $\widetilde{M}$ is a component of the preimage of a meridian of $\mathcal{E}$.  Hence the fibering which $\widetilde{\mathcal{E}}$ inherits from $\mathcal{E}'$ by pulling back using $q$ extends to a fibering of $\widetilde{M}$, which $q$ maps to that of $M'$ by construction.
\end{proof}

We will encounter the following situation: $M'$ is the trivial $F$-bundle over $S^1$ for some closed  surface $F$, homeomorphic to $F \times I/((x,1) \sim (x,0))$, and $(p')^{-1}(L)$ consists of simple closed curves in disjoint copies of $F$.  Here $I = [0,1]$.  Let $\pi \co M' \rightarrow F$ be projection to the first factor.  The second main result of this section describes a property of the collection $\pi((p')^{-1}(L))$ which allows a fibering of $M'$ to be found satisfying the hypotheses of Proposition \ref{technical_branched}.

\begin{prop}  Let $M = F \times I/(x,0) \sim (x,1)$, and suppose $L = \{\lambda_1, \hdots \lambda_m\}$ is a link in $M$ such that for each $j$ there exists $t_j \in I$ with $\lambda_j \subset F \times \{t_j\}$, and $t_j \ne t_{j'}$ for $j \ne j'$.  Suppose there is a collection of disjoint simple closed curves $\{\gamma_1,\hdots,\gamma_n\}$ on $F$, each transverse to $\pi(\lambda_j)$ for all $j$, with the following properties.  For each $j$, there is an $i$ such that $\pi(\lambda_j)$ intersects $\gamma_i$, and a choice of orientation of the $\gamma_i$ and all curves of $L$ may be fixed so that for any $i$ and $j$, $\gamma_i$ and $\pi(\lambda_j)$ have equal algebraic and geometric intersection numbers.  Then $M$ has a fibering transverse to $L$.  \label{transverse} \end{prop}

The other fiberings needed to prove this theorem may be found by \textit{spinning} annular neighborhoods of the $\gamma_i$ in the fiber direction.  We first saw this technique in \cite{Jaco}.

\begin{dfn}  Let $M = F \times I/(x,0) \sim (x,1)$ be the trivial bundle, and let $\gamma$ be a simple closed curve in $F$.  Let $A$ be a small annular neighborhood of $\gamma$, and fix a marking homeomorphism $\phi \co S^1 \times I \rightarrow A$.  We define the fibration obtained by spinning $A$ in the fiber direction to be
\[ F_A(t) = \left( (F-A) \times \{t\} \right)\ \bigcup\ \Phi_t(S^1 \times I), \]
where $\Phi_t(x,s) = (\phi(x,s),\rho(s)+t)$, for $t$ between $0$ and $1$.  Here we take $\rho \co I \rightarrow I$  to be a smooth, nondecreasing function taking $0$ to $0$ and $1$ to $1$, which is constant on small neighborhoods of $0$ and $1$ and has derivative at least $1$ on $[1/4,3/4]$.  

Given a collection of disjoint simple closed curves $\gamma_1,\hdots,\gamma_n$, one analogously produces a new fibration $F_{A_1,\hdots,A_n}(t)$ by spinning an annular neighborhood of each in the fiber direction.
\end{dfn}

Suppose $\lambda$ is a simple closed curve in $F$ which has identical geometric and algebraic intersection numbers with the core of each annulus $A_i$ in such a collection; that is, an orientation of $\lambda$ is chosen so that each oriented intersection with the core of each $A_i$ has positive sign.

\begin{lemma}  Let $\lambda$ be such a curve, embedded in $M$ by its inclusion into $F \times \{t_0\}$, $t_0 \in (0,1)$.  There is an ambient isotopy which moves $\lambda$ to be transverse to the fibration $F_{A_1,\hdots,A_n}(t)$, and which may be taken to be supported in an arbitrarily small neighborhood of $F \times \{t_0\}$.  \end{lemma}

\begin{proof} $\lambda$ may be isotoped in $F$ so that its intersection with the $A_i$ is of the form $(\{x_1\}\times I) \sqcup \hdots \sqcup (\{x_k\} \times I)$ for some collection $\{ x_1,\hdots,x_k\}$ of points in their cores.  For reference fix a Riemannian metric on $F$ in which the $A_i$ are isometrically embedded with their natural product metric, and choose a smooth unit--speed parametrization $\lambda(t)$ ($t \in I$) so that $\lambda([1/4,3/4]) = \{x_1\} \times [1/4,3/4]$.  For fixed small $\epsilon > 0$, we embed $\lambda$ in $M$ with the aid of a map $h_{\epsilon} \co I \rightarrow I$, defined as follows.  Let $h'_{\epsilon}$ be a smooth bump function which takes the value $-\epsilon$ on $[0,1/4]$ and $[3/4,1]$, is increasing on $[1/4,3/8]$ and decreasing on $[5/8,3/4]$, takes the value $2\epsilon$ on $[3/8,5/8]$, and has integral equal to $0$.  Then define $h_{\epsilon}$ by \[ h_{\epsilon}(s) = t_0 + \int_{0}^s h'_{\epsilon}, \]
and let $\lambda_{\epsilon}(s) = (\lambda(s),h_{\epsilon}(s))$.

At any point of $M$, the parametrization of $M$ as $F \times I/(x,1)\sim (x,0)$, gives a natural decomposition of the tangent space.  We call \textit{horizontal} the tangent planes to $F$, and let $\mathbf{t}$ denote the vertical vector pointing upward.  In the complement of the vertical tori determined by the $A_i \times I$, the new fiber surface $F_{A_1,\hdots,A_n}(t)$ has horizontal tangent planes, for each $t$.  Since the intersection of $\lambda_{\epsilon}$ with this region is contained in $\lambda_{\epsilon}([0,1/4] \cup [3/4,1])$, its tangent vector in this region has $\mathbf{t}$--component $-\epsilon$.  Hence intersections in this region between $\lambda_{\epsilon}$ and copies of the fiber surface $F_{A_1,\hdots,A_n}$ are transverse.

For points in $A_i$, consider the vertical plane spanned by $\mathbf{t}$ and the tangent vector to the $I$--factor of $A_i$.  Tangent vectors to $\lambda_{\epsilon}$ at points which lie in $A_i \times I$ lie in this plane with slope between $-\epsilon$ and $2\epsilon$, possibly greater than $-\epsilon$ only between $1/4$ and $3/4$.  On the other hand, the intersection of the tangent plane to $F_{A_1,\hdots,A_n}$ intersects the vertical plane in a line with slope greater than or equal to $0$, and greater than or equal to $1$ on $[1/4,3/4]$.  Thus as long as $2\epsilon < 1$, any intersection of $\lambda_{\epsilon}$ with a copy of $F_{A_1,\hdots,A_n}$ in these regions is transverse as well.  The original embedding of $\lambda$ may clearly be moved to $\lambda_{\epsilon}$ by a small ambient isotopy, and since $\lambda_{\epsilon}$ is transverse to each $F_{A_1,\hdots,A_n}(t)$, this proves the lemma.
\end{proof}

\begin{proof}[Proof of Proposition \ref{transverse}]  Let $\{\gamma_1,\hdots,\gamma_n\}$ be a collection satisfying the hypotheses, and let $\{A_i\}$ be a collection of disjoint annular regular neighborhoods of the $\gamma_i$ in $F$.  By the lemma above, each member $\lambda_j$ of $L$ may be moved by an ambient isotopy to be transverse to the fibration obtained by spinning each $A_i$ in the fiber direction.  Since the components of $L$ lie in disjoint fibers of the original fibration, these isotopies may be taken to have disjoint supports.  Then the inverse of their composition, applied to the fibration obtained by spinning $A$ in the fiber direction, produces a new fibration which is transverse to $L$.  \end{proof}


\section{Introducing the tetrus} \label{sec:tetrus}

In this section and the next, we will frequently encounter branched coverings of the form $q \co M_n \to M$, where $M_n$ and $M$ are manifolds with nonempty boundary.  In this case, the branch locus may have components which are properly embedded arcs in $M$.  If $T$ is the branch locus, we require that the regular neighborhood $\mathcal{N}(T)$ have the property that on a component $V$ of $q^{-1}(\mathcal{N}(T))$ projecting to a neighborhood of an arc component of $T$, there is a homeomorphism $V \to D^2 \times I$ such that $q$ is modeled by the product of the $k$-fold branched cover $D^2 \to D^2$ with the identity map on $I$.  With the same requirement on circle components of $T$, and $\mathcal{E}(T)$ defined as before, we remark that Lemma \ref{here's yer cover} holds verbatim in this context.

Thurston constructed a hyperbolic manifold with totally geodesic boundary, which he called the ``tripus,'' from two hyperbolic truncated tetrahedra in Chapter 3 of his notes \cite{Th1}.  A description of the tripus as the complement of a genus two handlebody embedded in $S^3$ may be found there.  In \cite{ZP}, Paoluzzi-Zimmermann generalized this construction, constructing for each $n \geq 3$ and $k$ between $0$ and $n-1$ with $(2-k,n)=1$ a hyperbolic manifold $M_{n,k}$ with geodesic boundary, for which the tripus is $M_{3,1}$.  Ushijima extended Thurston's description of the tripus as a handlebody complement in $S^3$ to show that each $M_{n,1}$, $n \geq 3$, is homeomorphic to the exterior --- that is, the complement of a regular neighborhood --- of Suzuki's Brunnian graph $\theta_n$ \cite{Ush}.  In particular, the tetrus $M_{4,1}$ is the exterior of the graph $\theta_4$ pictured in Figure \ref{tetrus}.

\begin{figure}[ht]
\begin{center}
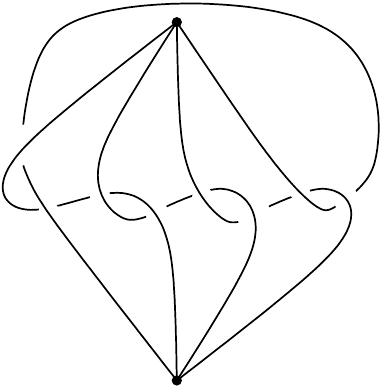
\end{center}
\caption{$\theta_4$}
\label{tetrus}
\end{figure}

There is an order--4 automorphism of the tetrus visible in the figure as a rotation through a vertical axis intersecting $\theta_4$ only in its two vertices.  The quotient of $M_{4,1}$ by the group that this automorphism generates yields a branched cover $q_{4,1} \co M_{4,1} \to B^3$, branched over the tangle $T$ pictured in Figure \ref{O_ntangle}.  In fact Paoluzzi-Zimmermann describe, for each $n \geq 3$ and $k$ with $0 \leq k < n$ and $(2-k, n) = 1$, an $n$-fold branched cover $q_{n,k} \co M_{n,k} \to B^3$, branched over $T$ (see \cite[Figures 4 \& 5]{ZP}).  The quotient map $q_{n,k}$ may be realized by a local isometry to an \textit{orbifold} $O_n$ with geodesic boundary, with underlying topological space $B^3$ and singular locus $T$ with strings of cone angle $2\pi/n$.  We summarize Paoluzzi-Zimmermann's description of the orbifold fundamental group of $O_n$ and its relationship with the fundamental groups of the $M_{n,k}$ in the theorem below; this collects various results in \cite{ZP}.

\begin{thm1}[Paoluzzi-Zimmermann] For each $n \geq 3$, the orbifold fundamental group of $O_n$ is presented as
\[ E_n \cong \langle\,X_n,H_n\ |\ H_n^n = (H_nX_nH_nX_n^{-2})^n = 1\,\rangle. \]
Each elliptic element of $E_n$ is conjugate to exactly one of $H_n$ or $H_nX_nH_nX_n^{-2}$.  For $(2-k,n)=1$, the fundamental group $G_{n,k} := \pi_1(M_{n,k})$ is the kernel of the projection $\pi_k \co E_n \onto \mathbb{Z}_n = \langle h_n \rangle$ given by $\pi_k(X_n)=h_n^k$ and $\pi_k(H_n)=h_n$.  \end{thm1}

Here we have departed from the notation of Paoluzzi-Zimmermann in distinguishing between presentations of the orbifold fundamental group of $O_n$ for different $n$, so that Paoluzzi-Zimmermann's ``$x$'' is replaced above by our ``$X_n$'' and similarly for ``$h$'' in their presentation for $E_n$ in \cite[p. 120]{ZP}.  Also, we have renamed the generator of $\mathbb{Z}_n$ to $h_n$.

\begin{figure}[ht]
\begin{center}
\input{O_ntangle.pdf_t.txt}
\end{center}
\caption{A tangle $T$ in the ball $B^3$}
\label{O_ntangle}
\end{figure}
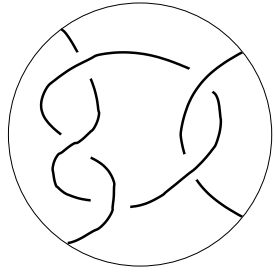

In view of Proposition \ref{fibered} of the previous section, it is important that we have a description of the fundamental group of the exterior of $T$ in $B^3$.  Let $\mathcal{N}(T)$ be a regular neighborhood of $T$ in $B^3$.  Then $\mathcal{N}(T)$ has two components, each homeomorphic to $D^2 \times I$ in such a way that its intersection with $T$ is sent to $\{(0,0)\} \times I$ and its intersection with $\partial B^3$ to $D^2 \times \{0,1\}$.  Then take $\mathcal{E}(T) = \overline{B^3 - \mathcal{N}(T)}$, the exterior of $T$ in $B^3$, and let $\pi(T) = \pi_1(\mathcal{E}(T))$.  We refer by a \textit{meridian of $T$} to a curve on $\partial \mathcal{N}(T) \cap \mathcal{E}(T)$ of the form $\partial D^2 \times \{y\}$, for $y \in I$.  Below we summarize some facts about $\pi(T)$, which may be found for instance in joint work with Eric Chesebro \cite[\S 2]{CD}.

\begin{lemma} \label{CD stuff}  The group $\pi(T)$ is free on generators $x$ and $h$.  In $\pi(T)$, the meridians of $T$ are represented by $h$ and $hxhx^{-2}$, and the four-holed sphere $\partial B^3 \cap \mathcal{E}(T)$ is represented by $\Lambda = \langle h, hxhx^{-2}, (xhx)h^{-1}(xhx)^{-1} \rangle$ in $\pi(T)$.  \end{lemma}

Using Lemma \ref{CD stuff}, we reinterpret the Theorem of Paoluzzi-Zimmermann below in a way that matches our treatment of branched covers in Section \ref{sec:Thurston}.

\begin{lemma} \label{M_nkbranching}  For each $n \geq 3$ and $k$ with $(2-k,n)=1$, let $\Gamma_{n,k}$ be the kernel of the map $\pi_{n,k} \co \pi(T) \onto \mathbb{Z}_n=\langle h_n \rangle$ given by $x \mapsto h_n^k$, $h \mapsto h_n$, and let $q_{n,k} \co \mathcal{E}_{n,k} \rightarrow \mathcal{E}(T)$ be the cover corresponding to $\Gamma_{n,k}$.  Then $q_{n,k}$ extends to the branched covering $q_{n,k} \co M_{n,k} \to B^3$ described by Paoluzzi-Zimmermann, where $M_{n,k}$ is obtained from $\mathcal{E}_{n,k}$ by filling each component of the preimage of $\partial \mathcal{N}(T) \cap \mathcal{E}(T)$ with a copy of $D^2 \times I$.  \end{lemma}

\begin{proof}  The homomorphism $\pi_{n,k}$ described in the lemma takes the elements $h$ and $hxhx^{-2}$ to $h_n$ and $(h_n)^{2-k}$, respectively, each of which generates $\mathbb{Z}_n$ when $2-k$ is relatively prime to $n$.  Since the meridians of $\mathcal{E}(T)$ are represented by $h$ and $hxhx^{-2}$, it follows that each meridian has connected preimage in $\mathcal{E}_{n,k}$, represented in $\Gamma_{n,k}$ by $h^n$ and $(hxhx^{-2})^n$, respectively.

If $U$ is a component of $\partial \mathcal{N}(T) \cap \mathcal{E}(T)$, then it is homeomorphic to $\partial D^2 \times I$, and by the paragraph above $(q_{n,k})^{-1}(U)$ is a connected $n$-fold cover of $U$, modeled by the product of the $n$-fold cover $\partial D^2 \to \partial D^2$ with the identity map $I \to I$.  Thus after filling each component of $\partial \mathcal{N}(T) \cap \mathcal{E}(T)$ and its preimage under $q_{n,k}$ with a cylinder $D^2 \times I$, $q_{n,k}$ extends to a branched cover modeled on the cylinders by the product of the standard $n$-fold branched cover $D^2 \to D^2$ with the identity map $I \to I$.

By our descriptions of $\mathcal{N}(T)$ and $\mathcal{E}(T)$, the image of this branched cover is homeomorphic to $B^3$, with branching locus $T$.  We claim that the domain is homeomorphic to $M_{n,k}$.  There is a map from $\pi(T)$ onto the orbifold group $E_n$ given by sending $x$ and $h$ to $X_n$ and $H_n$, respectively.  Then $\pi_{n,k}$ factors as this projection followed by the map $\pi_k$ defined by Paoluzzi-Zimmermann.  The description of $E_n$ thus implies that $G_{n,k}$ is the quotient of $\Gamma_{n,k}$ by the normal closure of $\{h^n,(hxhx^{-2})^n\}$, and the claim follows.
\end{proof}

The \textit{double} branched cover $M_2 \to B^3$, branched over $T$, was not addressed in \cite{ZP} since it does not admit the structure of a hyperbolic manifold with totally geodesic boundary.  In fact, it is homeomorphic to the trefoil knot exterior, and we will describe its Seifert fibered structure and the preimage of $L$ in detail in Section \ref{sec:fibering}.  Below we give a description consistent with that of Lemma \ref{M_nkbranching}.

\begin{lemma} \label{factors through}  Let $\Gamma_2$ be the kernel of the map $\pi_2 \co \pi(T) \onto \mathbb{Z}_2 = \{0,1\}$ given by $x \mapsto 1$, $h \mapsto 1$, and let $q_2 \co \mathcal{E}_2 \to \mathcal{E}(T)$ be the cover corresponding to $\Gamma_2$.  Then $q_2$ extends to the unique twofold branched cover $q_2 \co M_2 \rightarrow B^3$, after filling components of the preimage of $\partial \mathcal{N}(T) \cap \mathcal{E}(T)$ with copies of $D^2 \times I$.  \end{lemma}

\begin{proof}  Let $q \co M \rightarrow B^3$ be a twofold branched cover with branch locus $T$.  The associated cover $q \co \mathcal{E} \rightarrow \mathcal{E}(T)$ corresponds to a subgroup $\Gamma < \pi(T)$ which is of index 2 and hence normal.  Each element of $\pi(T)$ representing a meridian of $T$ must map nontrivially under the quotient $\pi(T) \to \pi(T)/\Gamma_2 \simeq \mathbb{Z}_2$, since $q$ branches nontrivially over each component of $T$.  By the description in Lemma \ref{CD stuff}, it follows that $h$ and $h^2x^{-1}$ map nontrivially, hence that each of $h$ and $x$ map to the generator.  Thus the only twofold branched cover of $B^3$, branched over $T$, is $q_2 \co M_2 \to B^3$ as described in the lemma.  \end{proof}

The uniqueness of $q_2 \co M_2 \rightarrow B^3$ implies that each $q_{n,k} \co M_{n,k} \to B^3$, where $n$ is even, factors through $q_2$.  We will take advantage of this fact below.


\section{Virtually fibering the doubled tetrus} \label{sec:fibering}

To construct a fibered cover for $DM_{4,1}$ using the methods of Section \ref{sec:Thurston}, we first describe a trivially fibered cover $p' \co M' \rightarrow M_2$ and in it, the preimage of the tangle $T$ of Figure \ref{O_ntangle}.  This uses a description of $T$ as a sum of \textit{rational tangles}, introduced by Conway \cite{Conway}.  $T$ is represented as $3\,0 + 2\,0$ in Conway's notation (see \cite[Fig. 1,2,3]{Conway}), where $3\,0$ and $2\,0$ associate to the rational numbers $1/3$ and $1/2$, respectively.  We will refer to $T$ as the \textit{Montesinos tangle} $T(1/3,1/2)$, referring to Montesinos' construction \cite{Mont} of twofold branched covers of $S^3$, branched over links built as sums of rational tangles.  Below we give an ad hoc version of this construction which suits our purposes.

Let $V = D^2 \times S^1$ be the solid torus, embedded in $\mathbb{C}^2$ as the cartesian product of the unit disk in $\mathbb{C}$ with its boundary, and oriented as a product of the standard orientation on $D^2$ with the boundary orientation $S^1$ inherits as $\partial D^2$.  Define the \textit{complex conjugation-induced} involution of $V$ by $(z,w) \mapsto (\bar{z},\bar{w})$.  The fixed set is $S = \{(r,\pm 1)\ |\ r \in [-1,1]\,\}$, a disjoint union of two arcs properly embedded in $V$.  The quotient map $q \co V \rightarrow V/((z,w) \sim (\bar{z},\bar{w})) \cong B^3$ is a twofold branched covering with branching locus $S$.  Each meridian disk $D^2 \times \{\pm1\}$ of $V$ is mapped by $q$ to a disk in $B^3$ which determines an isotopy rel endpoints between an arc of $q(S)$ and an arc on $\partial B^3$.  That is, $q(S)$ is the trivial two--string tangle $B^3$.

A rational number $p/q$ in lowest terms determines a Seifert fibering of $V$, with an \textit{exceptional fiber} parametrized by $\gamma_0(t)= (0,e^{2\pi it})$, $t \in I$, and \textit{regular fibers} parametrized by 
\begin{align}  \label{fiber param}  \gamma_z(t) = \left(z\,e^{2\pi i\cdot pt},e^{2\pi i \cdot qt} \right),\ \ t \in I, \end{align}
for $z \in D^2-\{0\}$.  Then $z$ and $w$ in $D^2$ determine the same fiber if and only if $w = ze^{2\pi ik/q}$ for some $k \in \mathbb{Z}$.  Hence for any $w \in S^1$, the quotient map sending each fiber to a point restricts on $D^2 \times \{w\}$ to a $q$-fold branched covering, branched at the origin.  We let $V_{p/q}$ denote $V$ equipped with this fibering, and divide $\partial V_{p/q}$ into annuli $A_{p/q}$ and $B_{p/q}$, parametrized as follows.  Define a model annulus $A = I \times I /(x,0) \sim (x,1)$, inheriting a ``vertical'' fibering by circles from arcs $\{x\} \times I$, a ``horizontal'' fibering from arcs of the form $I \times \{y\}$, and an orientation from the standard orientation on $I \times I$.  Define $\phi_{p/q} \co A \rightarrow \partial V_{p/q}$ by \begin{align*}
  \phi_{p/q}(x,y) = \left(e^{2\pi i \left(\frac{1-2x}{4q}\right)}e^{2\pi i\cdot py}, e^{2\pi i\cdot qy} \right).  \end{align*}
Choose $a$ and $b$ such that $ap+bq = 1$, and define $\psi_{p/q} \co A \to \partial V_{p/q}$ by \begin{align*}
  \psi_{p/q}(x,y) = \left(e^{2\pi i \left(\frac{2x+1}{4q}\right)}e^{2\pi i \cdot p\left(y-\frac{a}{2q}\right)}, e^{2\pi i \cdot q\left(y-\frac{a}{2q}\right)}\right).
\end{align*}  
 Then $\phi_{p/q}$ and $\psi_{p/q}$ have the following properties.  \begin{enumerate}
  \item Taking $A_{p/q} \doteq \phi_{p/q}(A)$ and $B_{p/q} \doteq \psi_{p/q}(A)$, we have $A_{p/q} \cup B_{p/q} = \partial V_{p/q}$, and $A_{p/q} \cap B_{p/q} = \phi_{p/q}(\partial A) = \psi_{p/q}(\partial A)$.
  \item  Each of $\phi_{p/q}$ and $\psi_{p/q}$ takes a vertical fiber of $A$ to a Seifert fiber of $V_{p/q}$ and a horizontal fiber to a closed arc in $\partial D^2 \times \{y\}$ for some $y \in S^1$.
  \item  Giving $A_{p/q}$ and $B_{p/q}$ the boundary orientations from $V_{p/q}$, $\phi_{p/q}$ reverses and $\psi_{p/q}$ preserves orientation.
  \item  The complex conjugation-induced involution on $V_{p/q}$ commutes with the map $(x,y) \mapsto (1-x,1-y)$ under each of $\phi_{p/q}$ and $\psi_{p/q}$.  In particular, each of $A_{p/q}$ and $B_{p/q}$ contains two points of $S \cap \partial V_{p/q}$, and $(1,1) \in A_{p/q}$.  \end{enumerate}

\begin{figure}
\begin{center}
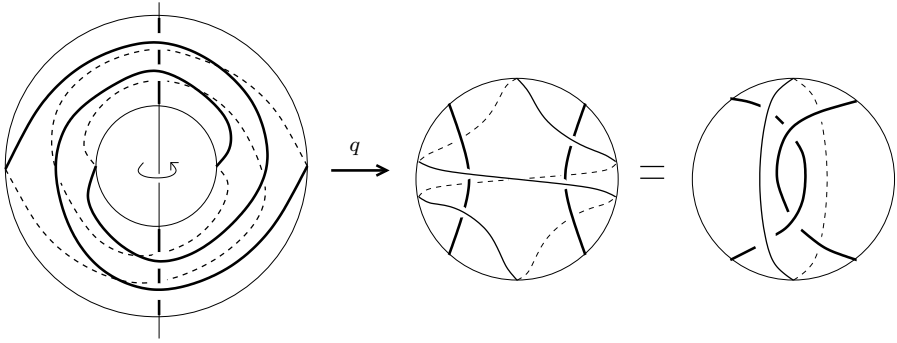
\end{center}
\caption{The double branched cover of the rational tangle $1/2$.}
\label{doublebranched}
\end{figure}

The map $q$ determines a double branched cover of $V_{1/2}$ to the ball, branched over the rational tangle $1/2$, as illustrated in Figure \ref{doublebranched}.  In the figure, the two parallel simple closed curves comprising $A_{1/2} \cap B_{1/2}$ are drawn on $\partial V_{1/2}$, projecting to the boundary of the indicated disk on $B^3$.  A similar picture holds for the double branched cover of $V_{1/3}$ to the $1/3$ rational tangle.  An appeal to Property 4 of the parametrizations above thus yields the following lemma.

\begin{lemma}  \label{twofold}  Define $M_2 = V_{1/3} \cup_{\phi_2 \psi_3^{-1}} V_{1/2}$.  There is a branched cover $q_2 \co M_2 \rightarrow B^3$, branched over $T$, which restricts on each $V_i$ to $q$.  \end{lemma}

Property 2 of the parametrizations $\psi_3$ and $\phi_2$ implies that $M_2$ inherits the structure of a Seifert fibered space from $V_{1/3}$ and $V_{1/2}$.  The lemma below describes a foliation of $M_2$ by surfaces, each meeting each Seifert fiber transversely.

\begin{lemma} \label{hor sfc}  Let $H_m = D^2 \times \{e^{2\pi i\frac{m-1}{2}}\} \subset V_3$, $m = 0,1$, and $S_n = D^2 \times \{e^{2\pi i \frac{n}{3}}\} \subset V_2$, $n = 0,1,2$.  Then $F = \left( \bigcup H_m \right) \cup \left( \bigcup S_n \right)$ is a connected surface homeomorphic to a one-holed torus, which is a fiber in the fibration of $M_2$ that restricts on $V_{1/3}$ or $V_{1/2}$ to the foliation by disks $D^2 \times \{y\}$.  A map $\sigma \co F \rightarrow F$ is determined by the following combinatorial data: $\sigma(H_m) = H_{1-m}$ for $m = 0,1$, and $\sigma(D_n) = D_{n+1}$ for $n = 0,1,2$ (take $n+1$ modulo three), so that $M \cong F \times [0,1]/((x,0) \sim (\sigma(x),1))$. \end{lemma}

\begin{proof}  By Property 2 of $\phi_{1/2}$ and $\psi_{1/3}$, for each $y \in S^1$, the gluing map $\phi_{1/2} \psi_{1/3}^{-1}$ takes components of $\partial D^2 \times \{y\} \cap B_{1/3} \subset V_{1/3}$ to components of $\partial D^2 \times \{y'\} \cap A_{1/2} \subset V_{1/2}$, for $y'$ determined by $y$.  It follows that the foliations of $V_{1/3}$ and $V_{1/2}$ by disks of the form $D^2 \times \{y\}$ join in $M_2$ to yield a foliation by surfaces.  We take $F$ to be the surface in this foliation containing $D^2 \times \{1\}$ in $V_{1/3}$, and illustrate its combinatorics in Figure \ref{horsfc}.  

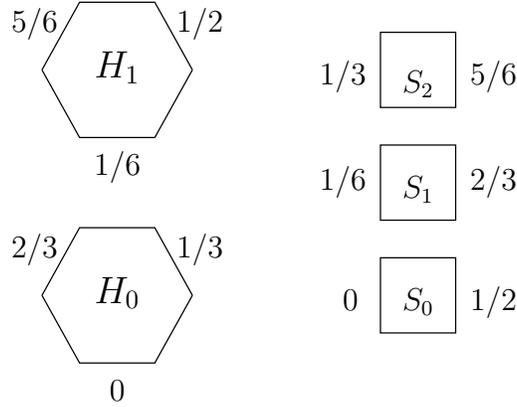
\begin{figure}[ht]
\begin{center}
\input{hor_sfc.pdf_t.txt}
\end{center}
\caption{The surface $F$ in $M_2$.}
\label{horsfc}
\end{figure}

We depict the disks $H_m$ as hexagons, since each intersects each of $A_{1/3}$ and $B_{1/3}$ in three arcs of its boundary.  Applying $\psi_{1/3}^{-1}$ to a component of $H_m \cap B_{1/3}$ yields an arc of the form $I \times \{h\}$ for some $h \in I$; in the figure, we have labeled each arc of $H_m \cap B_{1/3}$ by the corresponding $h$.  Each ``square'' $S_n$ intersects $A_{1/2}$ in two arcs of its boundary, labeled in Figure \ref{horsfc} by the height of their images under $\phi_{1/2}^{-1}$.  We assign each $H_m$ or $S_n$ the standard orientation from $D^2$, and picture them in Figure \ref{horsfc} with the orientation inherited from the page.  Then each labeled edge of $H_m$ is identified to that of $S_n$ with the same label, in  orientation--reversing fashion, by $\phi_{1/2}\psi_{1/3}^{-1}$.  Their union $F$ is now easily identified as a one--holed torus.  

Since the Seifert fibers of $V_{1/3}$ and $V_{1/2}$ are transverse to the disks $D^2 \times \{y\}$, Seifert fibers of $M_2$ transversely intersect each surface in the foliation described above.  There is a quotient map $\pi$, taking $M_2$ to a closed one-manifold (that is, $S^1$), determined by crushing each surface in the foliation described above to a point.   Then $F$ is the preimage under $\pi$ of a point in $S^1$, so cutting $M_2$ along $F$ yields a surface bundle over $I$, necessarily of the form $F \times I$.  With $F$ oriented as prescribed above, the normal orientation to $F$ in $M_2$ is the upward direction along Seifert fibers parameterized as in (\ref{fiber param}).  The ``first return map'' $\sigma \co F \to F$ is prescribed as follows: from $x \in F$, move in the normal direction along the Seifert fiber through $x$ until it again intersects $F$; the point of intersection is $\sigma(x)$.  It describes a monodromy for the description of $M_2$ as a bundle over $S^1$.

From the description of $\sigma$ and the fiber-preserving parameterizations $\psi_{1/3}$ and $\phi_{1/2}$, we find that for a point $x$ on an arc labeled by $h$ in Figure \ref{horsfc}, $\sigma(x)$ is the corresponding point on the arc labeled by $h+1/6$ (modulo $1$).  This models the behavior of $\sigma(x)$ on the regular fibers.  The intersection of $F$ with the singular fiber in $V_{1/3}$ is the disjoint union of the centers of $H_0$ and $H_1$, which are thus interchanged by $\sigma$.  An analogous description holds for the intersection of $F$ with the singular fiber in $V_{1/2}$, yielding the description of $\sigma$ in the statement of the lemma.  Now since $V_{1/3}$ is cut by its intersection with $F$ into two cylinders, and $V_{1/2}$ is cut into three, and these are identified along vertical annuli in their boundaries by $\phi_{1/2}\psi_{1/3}^{-1}$ to form $M_2$ cut along $F$, the description of $M_2$ as a fiber bundle with fiber $F$ follows.  

By construction, the components of $q_2^{-1}(T)$ in $M_2$ intersect each of $V_{1/3}$ and $V_{1/2}$ in the fixed set of the complex conjugation-induced involution.  In $V_{1/3}$, the component $[-1,1] \times \{-1\}$ is contained in $H_0$, and $[-1,1] \times \{1\} \subset H_1$.  In $V_{1/2}$, the component $[-1,1] \times \{1\}$ is contained in $S_0$, and $\phi_{1/2}\psi_{1/3}^{-1}$ takes its endpoints to endpoints of the components in $V_{1/3}$.  \end{proof}

Note that $q_2^{-1}(T)$ intersects each of $V_{1/3}$ and $V_{1/2}$ in the fixed locus of the complex conjugation-induced involution.  Since the components of this locus lie in the disks $D^2 \times \{\pm 1\}$, it follows that $q_2^{-1}(T)$ lies in fibers of the fibration $M_2 \to S^1$ described in Lemma \ref{hor sfc}.  The description of Lemma \ref{hor sfc} also implies that $\sigma$ is periodic with order 6.  Therefore $M_2$ has a sixfold cover $p' \co M' \rightarrow M_2$ which is trivially fibered; that is, $M' \cong F \times S^1$, such that components of $(q_2 \circ p')^{-1}(T)$ lie in disjoint fibers.

The uniqueness property of Lemma \ref{factors through} implies that the manifolds $M_2$ described there and in Lemma \ref{twofold} are homeomorphic as branched covers.  In particular, the map $q_{4,1} \co M_{4,1} \to B^3$ described in Lemma \ref{M_nkbranching} factors through $q_2$ described in Lemma \ref{twofold}.  Let $q_{2,1} \co M_{4,1} \to M_2$ be the map such that $q_{4,1} = q_{2,1} \circ q_2$.  Then $q_{2,1}$ is a twofold branched cover, branched over $(q_2)^{-1}(T)$.  

Recall, from the first paragraph of Section \ref{sec:tetrus}, that Lemma \ref{here's yer cover} applies in the context of manifolds with nonempty boundary.  

\begin{lemma} \label{magic cover}  Let $\tilde{p} \co \widetilde{M} \to M_2$ be the map produced by Lemma \ref{here's yer cover}, such that $\tilde{p} = p' \circ q = q_{2,1} \circ p$ for a branched cover $q \co \tilde{M} \to M'$ and a cover $p \co \tilde{M} \to M_{4,1}$.  Then $p$ has degree $6$ and $\partial \widetilde{M}$ is connected, with genus $13$.  \end{lemma}

\begin{proof}  Recall that $\tilde{p}$ extends the covering map $\tilde{p} \co \widetilde{\mathcal{E}} \to \mathcal{E}_2$, obtained by completing the diamond of covers $q_{2,1} \co \mathcal{E}_{4,1} \to \mathcal{E}_2$ and $p' \co \mathcal{E}' \to \mathcal{E}_2$, after filling along components in $\widetilde{\mathcal{E}}$ of the preimage of meridians of $\mathcal{E}_2$.  Lemma \ref{factors through} implies that the map $\pi_2 \co \pi(T) \onto \mathbb{Z}_2$ takes the elements $h$ and $hxhx^{-2}$ representing the meridians of $T$ to the generator.  Hence each meridian of $T$ has connected preimage, the corresponding meridian in $\mathcal{E}_2$ of $(q_2)^{-1}(T)$, and these are represented in $\Gamma_2 \subset \pi(T)$ by $h^2$ and $(hxhx^{-2})^{2}$, respectively.

Since $p' \co M' \to M_2$ has degree $6$ and $q_{2,1} \co M_{4,1} \to M_2$ has degree $2$, the subgroups $\Gamma'$ and $\Gamma_{2,1}$ corresponding to $\mathcal{E}'$ and $\mathcal{E}_{2,1}$ have indices $6$ and $2$ in $\Gamma_2$, respectively.  Thus the subgroup $\widetilde{\Gamma} = \Gamma' \cap \Gamma_{4,1}$ corresponding to $\widetilde{\mathcal{E}}$ has index two in $\Gamma'$, unless $\Gamma_{4,1}$ contains $\Gamma'$.  But Lemma \ref{M_nkbranching} implies that meridians of $\mathcal{E}(T)$ map to the generator of $\mathbb{Z}_4$ under $\pi_{n,k}$; hence each has connected preimage in $\mathcal{E}_{4,1}$, represented in $\Gamma_{4,1}$ by $h^4$ or $(hxhx^{-2})^4$.  On the other hand, since the components of $(q_2)^{-1}(T)$ in $M_2$ lift to $M'$, the meridians of $(q_2)^{-1}(T)$ do as well.  It follows that $\Gamma'$ contains the elements $h^2$ and $(hxhx^{-2})^2$ and thus contains $\widetilde{\Gamma}$ with index $2$.  Then $[\Gamma_{4,1}:\widetilde{\Gamma}] = 6$, and this is the degree of $p \co \widetilde{M} \to M_{4,1}$.

The subgroup $\Lambda < \pi(T)$, representing $\mathcal{E}(T) \cap \partial B^3$, identified in Lemma \ref{CD stuff} contains the meridian representatives $h$ and $hxhx^{-2}$.  Hence the map $\pi_2$ determining $q_2 \co \mathcal{E}_2 \to \mathcal{E}(T)$ takes $\Lambda$ onto $\mathbb{Z}_2$, so $\mathcal{E}(T) \cap \partial B^3$ has connected preimage in $\mathcal{E}_2$.  This is the four holed subsurface $\partial \mathcal{E}_2 \cap \partial M_2$, and $\partial M_2$ is obtained from it by filling with disks of the form $D^2 \times \{\pm 1\}$ bounding components of $(q_2)^{-1}(\mathcal{N}(T))$.  Since $M'$ is homeomorphic to $F \times S^1$, it has a single boundary component, and since components of $(q_2)^{-1}(T)$ lift to $M'$, their meridians lift to $\mathcal{E}'$.  In particular, the subgroup $\Lambda' < \Gamma'$ corresponding to $\partial \mathcal{E}' \cap \partial M'$ contains the meridian representatives $h^2$ and $(hxhx^{-2})^2$ of $\Gamma_2$.  Since these elements are not contained in $\Gamma_{4,1}$, $\widetilde{\Lambda} = \Lambda' \cap \Gamma_{4,1}$ has index two in $\Lambda'$.  Therefore $\partial \mathcal{E}' \cap \partial M'$ has connected preimage in $\widetilde{\mathcal{E}}$; this is the surface $\widetilde{\mathcal{E}} \cap \partial \widetilde{M}$.  It follows that $\partial \widetilde{M}$ is connected.

The boundary of the tetrus, $\partial M_{4,1}$, has genus $3$, which can be verified by the fact that it is a fourfold branched cover of $\partial B^3$, branched over the four points $T \cap \partial B^3$.  Since $\partial \widetilde{M} = p^{-1}(\partial M_{4,1})$ is a connected sixfold cover, an Euler characteristic calculation shows that it has genus $13$.
\end{proof}

We now turn from consideration of the manifolds-with-boundary $M_{n,k}$ to their doubles, as defined at the beginning of the paper.  It is clear from the definition that a map $f \co M \to N$ between manifolds with boundary determines a map, which we again denote $f \co DM \to DN$, between doubles, and that the map between doubles inherits the property of being a cover or branched cover of degree $n$ from the original map.

\begin{prop} \label{doubled magic cover}  Taking the branched cover $q_{2,1} \co DM_{4,1} \to DM_2$ and the cover $p' \co DM' \to DM_2$ to be determined by the corresponding maps on $M_{4,1}$ and $M'$, respectively, the map supplied by Lemma \ref{here's yer cover} is $\tilde{p} \co D\widetilde{M} \to DM_2$.
\end{prop}

The point of this proposition is that Lemma \ref{here's yer cover} is ``doubling equivariant"; that is, applying it and then doubling the resulting diagram of (branched) covers yields the same result as doubling first and then applying it.  It is a consequence of the normal form theorem for free products with amalgamation.

\begin{fact}  Let $A$ and $B$ be groups sharing a subgroup $C$, and let $A *_{C} B$ be the free product of $A$ with $B$, amalgamated over $C$.  If $\pi \co A*_C B \to K$ is an epimorphism to a finite group $K$ such that $\pi(C) = K$, then 
$$\ker \pi = \langle \ker \pi |_A, \ker \pi |_{B} \rangle \simeq (\ker \pi |_A)*_{\ker \pi |_C} (\ker \pi |_{B}).$$  \end{fact}

\begin{proof}[Proof of Fact]  It is clear that $\langle \ker \pi |_A, \ker \pi |_B \rangle$ is contained in $\ker \pi$.  We claim equality; this follows from the normal form theorem for free products with amalgamation (see \cite[Ch. IV, \S 2]{LS}).  Fix sets $S_A$ and $S_B$ of right coset representatives for $C$ in $A$ and $B$, respectively.  Then the normal form theorem asserts that each $g \in A*_C B - \{1\}$ has a \textit{normal form}, which is a unique expression $ g = c s_1 s_2 \cdots s_n$ 
for some $n \geq 1$, where $c \in C$ and each $s_i$ is in $S_A$ or $S_B$, with $s_i \in S_A$ if and only if $s_{i+1} \in S_B$.  We will call $n$ the \textit{length} of $g$, and note that the claim is immediate if $g$ has length $1$.

If $g \in \ker \pi$ has length $n > 1$, then write $g$ in normal form as above, and let $c_0 \in C$ have the property that $\pi(c_0) = \pi(c s_1 \cdots s_{n-1})$.  Taking $g_0 = c s_1 \cdots s_{n-1} c_0^{-1}$
and $g_1 = c_0 g_n$, we may write $g = g_0 g_1$ as a product of words in $\ker \pi$.  It is evident that $g_1$ has length $1$ and easily proved, by passing elements of $C$ to the left, that $g_0$ has length at most $n-1$.  Then by induction, each of $g_0$ and $g_1$ is in $\langle \ker \pi |_A, \ker \pi |_{B} \rangle$, so $g$ is as well.

The normal form theorem also implies that the naturally embedded subgroups $A$ and $B$ in $A*_C B$ intersect in $C$.  Therefore $\ker \pi |_A \cap \ker \pi |_{B} = \ker \pi |_C$, and it follows again from the normal form theorem that the inclusion-induced map $(\ker \pi |_A)*_{\ker \pi |_C} (\ker \pi |_{B}) \to \langle \ker \pi |_A, \ker \pi |_{B} \rangle$ is an isomorphism.  \end{proof}

\begin{proof}[Proof of Proposition \ref{doubled magic cover}]  Recall that we have identified the exterior of $T$, $\mathcal{E}(T)$, with $\overline{B^3 - \mathcal{N}(T)}$, where $\mathcal{N}(T)$ is a regular neighborhood of $T$ in $B^3$, with two components homeomorphic to $D^2 \times I$.  The double of $B^3$ is homeomorphic to $S^3$, $T$ doubles yielding the link $L$ of Figure \ref{doubledtangle}, and $\mathcal{N}(T)$ doubles yielding a regular neighborhood $\mathcal{N}(L)$.  Each component of $\mathcal{N}(L)$ is homeomorphic to $D^2 \times S^1$, with the corresponding component of $\mathcal{N}(T)$, homeomorphic to $D^2 \times I$, mapping in by $(x,y) \mapsto (x,e^{\pi i y})$ and its mirror image by $(x,y) \mapsto (x,e^{-\pi i y})$.  Using this description of $\mathcal{N}(L)$, the link exterior $\mathcal{E}(L)$ is the double of $\mathcal{E}(T)$ across the subsurface $\partial B^3 \cap \mathcal{E}(T)$ of $\partial \mathcal{E}(T)$, and meridians of $T$ in $\mathcal{E}(T)$ are meridians of $L$ in $\mathcal{E}(L)$.  

Using the above description, van Kampen's theorem describes $\pi(L)$ as a free product with amalgamation, $\pi(L) \simeq \pi(T) *_{\Lambda}\overline{\pi(T)}$, across the subgroup $\Lambda$ from  Lemma \ref{CD stuff}, which corresponds to $\partial B^3 \cap \mathcal{E}(T)$.  Here $\overline{\pi(T)} = \{ \bar{g}\,|\, g \in \pi(T) \}$ is $\pi_1(\overline{\mathcal{E}(T)})$, isomorphic to $\pi_T$.  There is a ``doubling involution" $r$ on $\pi(L)$ determined by $r(g) = \bar{g}$, $g \in \pi(T)$.  This is induced by the doubling involution exchanging $\mathcal{E}(T)$ with $\overline{\mathcal{E}(T)}$ in $\mathcal{E}(L)$, fixing their intersection $\partial B^3 \cap \partial \mathcal{E}(T)$.  

The projections $\pi_{n,k} \co \pi(T) \to \mathbb{Z}_n$  (respectively, $\pi_2 \co \pi(T) \to \mathbb{Z}_2$ )defined in Lemma \ref{M_nkbranching} (resp. Lemma \ref{factors through}) uniquely determine corresponding projections on $\pi(L)$ with the property that $\pi_{n,k} \circ r = \pi_{n,k}$ (resp. $\pi_2 \circ r = \pi_2$).  Since $\Lambda$ contains the meridian representatives $h$ and $hxhx^{-2}$, each such projection maps it onto the image of $\pi(L)$.  Then by the fact above, we have 
$$ D\Gamma_{n,k} \doteq \ker \pi_{n,k} = \langle \Gamma_{n,k}, \overline{\Gamma}_{n,k} \rangle \simeq \Gamma_{n,k} *_{\Lambda_{n,k}} \overline{\Gamma}_{n,k}, $$
where $\Lambda_{n,k} \doteq \ker \pi_{n,k} |_{\Lambda}$.  (The corresponding fact holds for $D\Gamma_2 \doteq \ker \pi_2$.)  

We now define $D\mathcal{E}_{n,k}$ (respectively, $D\mathcal{E}_2$) to be the double of $\mathcal{E}_{n,k}$ (resp. $\mathcal{E}_2$) across the subsurface $\partial \mathcal{E}_{n,k} \cap \partial M_{n,k} \subset \partial \mathcal{E}_{n,k}$ (resp. $\partial \mathcal{E}_2 \cap \partial M_2$).  This notation is somewhat abusive, since this subsurface does not occupy all of $\partial \mathcal{E}_{n,k}$, but we note that it is represented in $\Gamma_{n,k}$ by $\Lambda_{n,k}$, since its complement is $(q_{n,k})^{-1}(\partial \mathcal{N}(T) \cap \mathcal{E}(T))$.  Then $D\mathcal{E}_{n,k}$ (respectively, $D\mathcal{E}_2$) covers $\mathcal{E}(L)$, and the description above makes clear that this is the cover corresponding to $D\Gamma_{n,k}$ (resp. $D\Gamma_2$).  Hence filling $D\mathcal{E}_{n,k}$ (resp. $D\mathcal{E}_2$) along preimages of meridians of $L$ yields the branched cover $DM_{n,k}$ (resp. $DM_2$) defined in the statement of the proposition.  

We make a similar claim regarding $p' \co DM' \to DM_2$.  The cover $p' \co \mathcal{E}' \to \mathcal{E}_2$ is regular, corresponding to the kernel of a map $\pi' \co \Gamma_2 \to \mathbb{Z}_6$, and since $\partial \mathcal{E}' \cap \partial M'$ is connected, it corresponds to a subgroup $\Lambda' = \ker \pi'|_{\Lambda_2}$ of index $6$ in $\Lambda_2$.  Then defining $\pi' \co \Gamma_2 *_{\Lambda_2} \overline{\Gamma}_2 \to \mathbb{Z}_6$ by requiring $\pi' \circ r = \pi'$, we argue as above to show that $p' \co DM' \to DM_2$ is obtained by filling the cover of $D\mathcal{E}_2$ corresponding to $\ker \pi'$ along preimages of meridians.  

We recall from the proof of Lemma \ref{magic cover} that $\widetilde{\Gamma} = \Gamma_{4,1} \cap \Gamma'$ has index $2$ in $\Gamma'$, and that $\widetilde{\Lambda} = \Gamma_{4,1} \cap \Lambda'$ has index $2$ in $\Lambda'$.  Then it follows as above that the subgroup $D\widetilde{\Gamma} \doteq D\Gamma_{4,1} \cap D\Gamma' = \langle \widetilde{\Gamma}, r(\widetilde{\Gamma}) \rangle$, and that it is isomorphic to the free product of $\widetilde{\Gamma}$ with itself, amalgamated across $\widetilde{\Lambda}$.  The proposition follows.
\end{proof}

Proposition \ref{doubled magic cover} supplies a diamond of maps to which Proposition \ref{technical_branched} may be applied.  We thus prove Theorem \ref{fibered} below by using Proposition \ref{transverse} to find a fibering of $M'$ transverse to the preimage of $L$.

\begin{proof}[Proof of Theorem \ref{fibered}]  By Lemma \ref{hor sfc}, $M_2$ is homeomorphic to a bundle over $S^1$ with fiber the surface $F$ depicted in Figure \ref{horsfc} and monodromy $\sigma \co F \to F$.  The fibers of $M_2$ join to yield a fibering of $DM'$ with fiber surface $DF$, the double of $F$, and monodromy map which we will call $D\sigma$.  This is pictured in Figure \ref{doubledsfc}.  Here the hexagons and squares to the left of the vertical line are fitted together along labeled arcs as in Figure \ref{horsfc} forming a copy of $F$, and the hexagons and squares to the right of the vertical line are fitted together along their labeled edges forming a copy of $\overline{F}$.  To form $DF$, each unlabeled edge is identified with its correspondent by reflection through the vertical line.  $D\sigma$ is the map that restricts on $F$ to $\sigma$ and is equivariant with respect to the doubling involution, hence is itself of order $6$.  

Recall that $p' \co M' \to M_2$, defined below Lemma \ref{hor sfc}, is the sixfold cover of $M_2$ which is  trivially fibered over $S^1$ with fiber $F$.  Since $DM'$ is the double of $M'$, it is trivially fibered over $S^1$ with fiber $DF$.  Using Lemma \ref{hor sfc}, we may identify $DM_2$ with $DF \times I/((x,0)\sim (D\sigma(x),1))$.  Then identifying $DM'$ with $DF \times I/((x,0) \sim (x,1))$, the covering map is given by $(x,y) \mapsto ((D\sigma)^k(x),6(y-\frac{k}{6}))$ for $ y \in [\frac{k}{6},\frac{k+1}{6}]$, $0 \leq k < 6$.  Since the components of $(q_2)^{-1}(T)$ lie in disjoint copies of the fiber surface $F$ for $M_2$,  the components of $q_2^{-1}(L)$ lie in disjoint copies of $DF \subset M$.  Take $\pi \co DM' \rightarrow DF$ to be projection onto the first factor.  Then by the description above, if $\lambda$ is a component of $(q_2\circ p')^{-1}(L) \subset M'$, $\pi(\lambda)$ is a simple closed curve on $DF$ and $D\sigma$ takes $\pi(\lambda)$ to $\pi(\lambda')$, where $\lambda'$ is another component of $(q_2\circ p')^{-1}(L)$.

To understand $(q_2 \circ p')^{-1}(L)$, we first describe $(q_2)^{-1}(T)$ in $M_2$.  This the fixed set of the involution which restricts on each of $V_{1/3}$ and $V_{1/2}$ to the complex conjugation-induced involution.  The fixed arc $[1,1] \times \{1\} \subset V_{1/3}$ lies in the disk $H_0$, running from the midpoint of the side labeled $0$ to the midpoint of the opposite, unlabeled side in Figure \ref{horsfc}.  The other arc $[-1,1]\times \{-1\} \subset V_{1/3}$ runs from the midpoint of the side of $H_1$ labeled $1/2$ to the midpoint of the opposite side.  The arc $[-1,1] \times \{1\} \subset V_{1/2}$ lies in $S_0$, joining the midpoints of the sides labeled $0$ and $1/2$.  Thus the union of these three arcs is an arc properly embedded in $M_2$, comprising one component of the preimage of $T$.  

The other component of $q_2^{-1}(T)$ is the arc $[-1,1] \times \{-1\} \subset V_{1/2}$, with endpoints in $B_{1/2} \subset \partial M_2$.  This lies in the disk $D^2 \times \{-1\}$ in $V_{1/2}$, midway between $S_1$ and $\sigma(S_1) = S_2$.  Thus using the description from Lemma \ref{hor sfc} of $M_2$ as $F \times I/(x,0) \sim (\sigma(x),1)$, the second arc of $q_2^{-1}(T)$ lies in $S_1 \times \{1/2\}$, joining midpoints of the unlabeled boundary components.  

Now in $DM_2$, $(q_2)^{-1}(L)$ is the double of $(q_2)^{-1}(T)$, with one component in $DF \times \{0\}$ and one in $DF \times \{1/2\}$.  Then $(q_2 \circ p')^{-1}(L) \subset DM'$ has twelve components in disjoint copies of $DF$.  This set is depicted by the dashed arcs in Figure \ref{doubledsfc}.  If $\alpha$ is such an arc, a simple closed curve on $DF$ containing $\alpha$ may be obtained by taking the union of a collection of arcs $\mathcal{A}$, of minimal cardinality such that $\alpha \in \mathcal{A}$ and for each $\beta \in \mathcal{A}$, the arcs meeting $\beta$ at its endpoints are also in $\mathcal{A}$.  Inspection reveals six such simple closed curves, permuted by $D\sigma$, each of the form $\pi(\lambda)$ for some component $\lambda$ of $(q_2 \circ p')^{-1}(L)$.  There are six, rather than twelve, because $(D\sigma)^{3}$ is the hyperelliptic involution of $DF$, which preserves simple closed curves.

\begin{figure}
\begin{center}
\input{doubled_sfc.pdf_t.txt}
\end{center}
\caption{The collection of transverse curves in $DF$}
\label{doubledsfc}
\end{figure}
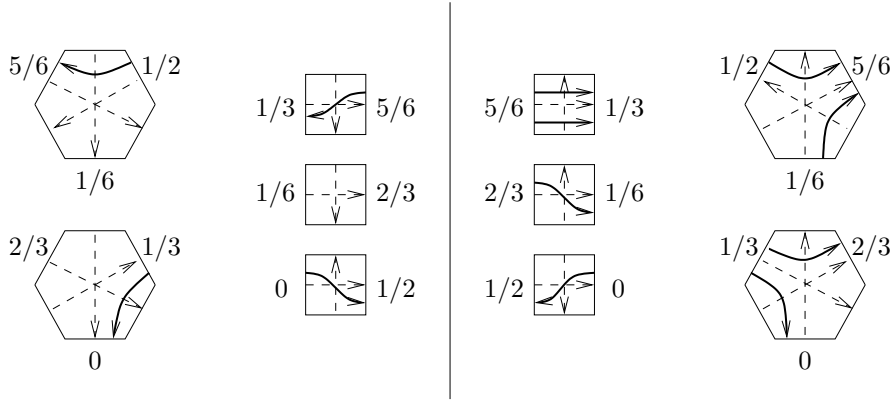

The bold arcs depicted in Figure \ref{doubledsfc} join to produce two disjoint simple closed curves in $DF$ transverse to this collection.  We have indicated orientations with arrows so that, giving $DF$ the standard orientation from the page, the algebraic and geometric intersection numbers of each bold curve with each dashed curve coincide.  Thus by Proposition \ref{transverse}, spinning annuli along the bold curves yields a fibration of $DM'$ transverse to $(q_2 \circ p')^{-1}(L)$.  Then by Proposition \ref{technical_branched}, $D\widetilde{M}$ is fibered, and the branched cover $q \co D\widetilde{M} \to DM'$ takes fibers to fibers.  

If $\widetilde{F}$ is the fiber surface of $D\widetilde{M}$, then $q$ restricts on $\widetilde{F}$ to a twofold branched cover of a fiber surface of $DM'$ transverse to $(q_2 \circ p')^{-1}(L)$, branched over their points of intersection.  From Figure \ref{doubledsfc}, we find $16$ points of intersection between the arcs of $\pi((q_2 \circ p')^{-1}(L))$ and the bold arcs which determine the spinning curves.  It follows that a spun fiber surface has $32$ points of intersection with $(q_2 \circ p')^{-1}(L)$, since $\pi$ maps these curves two-to-one.  Since the spun fiber surface of $DM'$ has genus two, an Euler characteristic calculation shows that $\widetilde{F}$ has genus $19$.  Together with the information in Lemma \ref{magic cover} and Proposition \ref{doubled magic cover}, this establishes the theorem.  \end{proof}


\bibliographystyle{plain}
\bibliography{fibering}

\end{document}

%% file: doubles.pdf_t.txt
\begin{picture}(0,0)%
\includegraphics{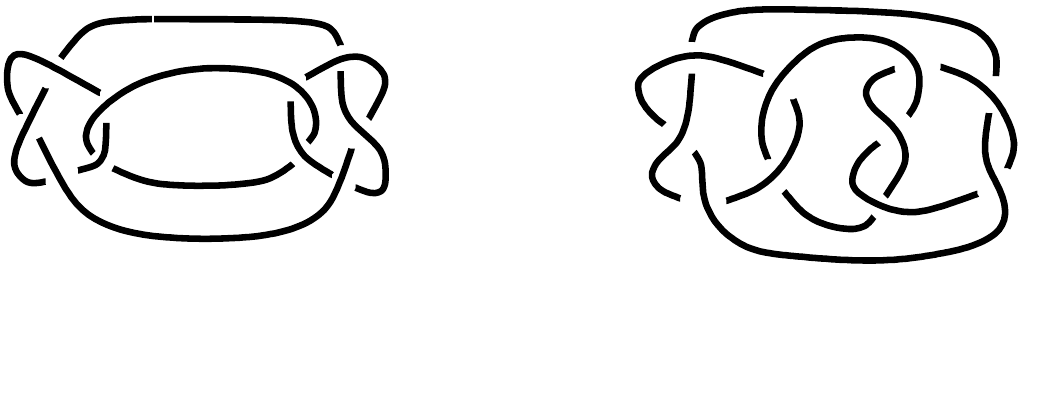}%
\end{picture}%
\setlength{\unitlength}{4144sp}%
\begingroup\makeatletter\ifx\SetFigFontNFSS\undefined%
\gdef\SetFigFontNFSS#1#2#3#4#5{%
  \reset@font\fontsize{#1}{#2pt}%
  \fontfamily{#3}\fontseries{#4}\fontshape{#5}%
  \selectfont}%
\fi\endgroup%
\begin{picture}(4761,1783)(1375,-1732)
\put(2206,-1636){\makebox(0,0)[b]{\smash{{\SetFigFontNFSS{14}{16.8}{\rmdefault}{\mddefault}{\updefault}{\color[rgb]{0,0,0}$L$}%
}}}}
\put(5311,-1636){\makebox(0,0)[b]{\smash{{\SetFigFontNFSS{14}{16.8}{\rmdefault}{\mddefault}{\updefault}{\color[rgb]{0,0,0}$L_{\phi}$}%
}}}}
\end{picture}%

%% file: theta_4.pdf_t
\begin{picture}(0,0)%
\includegraphics{theta_4.pdf}%
\end{picture}%
\setlength{\unitlength}{2155sp}%
\begingroup\makeatletter\ifx\SetFigFont\undefined%
\gdef\SetFigFont#1#2#3#4#5{%
  \reset@font\fontsize{#1}{#2pt}%
  \fontfamily{#3}\fontseries{#4}\fontshape{#5}%
  \selectfont}%
\fi\endgroup%
\begin{picture}(3350,3381)(2948,-3256)
\end{picture}%

%% file: O_ntangle.pdf_t.txt
\begin{picture}(0,0)%
\includegraphics{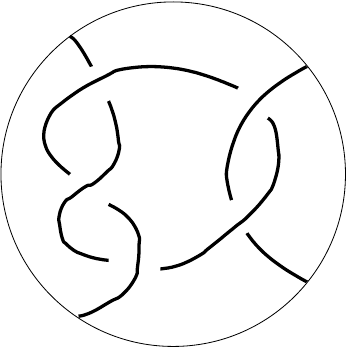}%
\end{picture}%
\setlength{\unitlength}{2072sp}%
\begingroup\makeatletter\ifx\SetFigFontNFSS\undefined%
\gdef\SetFigFontNFSS#1#2#3#4#5{%
  \reset@font\fontsize{#1}{#2pt}%
  \fontfamily{#3}\fontseries{#4}\fontshape{#5}%
  \selectfont}%
\fi\endgroup%
\begin{picture}(3166,3166)(2243,-3219)
\end{picture}%

%% file: branch_cov_eg.pdf_t
\begin{picture}(0,0)%
\includegraphics{branch_cov_eg.pdf}%
\end{picture}%
\setlength{\unitlength}{2155sp}%
\begingroup\makeatletter\ifx\SetFigFont\undefined%
\gdef\SetFigFont#1#2#3#4#5{%
  \reset@font\fontsize{#1}{#2pt}%
  \fontfamily{#3}\fontseries{#4}\fontshape{#5}%
  \selectfont}%
\fi\endgroup%
\begin{picture}(10268,3894)(226,-3493)
\put(4276,-1321){\makebox(0,0)[b]{\smash{{\SetFigFont{8}{9.6}{\rmdefault}{\mddefault}{\updefault}{\color[rgb]{0,0,0}$q$}%
}}}}
\end{picture}%

%% file: hor_sfc.pdf_t.txt
\begin{picture}(0,0)%
\includegraphics{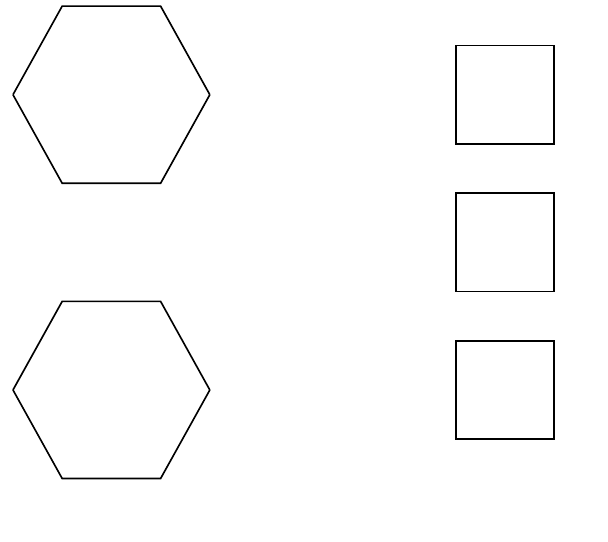}%
\end{picture}%
\setlength{\unitlength}{4144sp}%
\begingroup\makeatletter\ifx\SetFigFontNFSS\undefined%
\gdef\SetFigFontNFSS#1#2#3#4#5{%
  \reset@font\fontsize{#1}{#2pt}%
  \fontfamily{#3}\fontseries{#4}\fontshape{#5}%
  \selectfont}%
\fi\endgroup%
\begin{picture}(2775,2461)(3766,-3905)
\put(4276,-3841){\makebox(0,0)[b]{\smash{{\SetFigFontNFSS{12}{14.4}{\rmdefault}{\mddefault}{\updefault}{\color[rgb]{0,0,0}$0$}%
}}}}
\put(4276,-2491){\makebox(0,0)[b]{\smash{{\SetFigFontNFSS{12}{14.4}{\rmdefault}{\mddefault}{\updefault}{\color[rgb]{0,0,0}$1/6$}%
}}}}
\put(4771,-1636){\makebox(0,0)[b]{\smash{{\SetFigFontNFSS{12}{14.4}{\rmdefault}{\mddefault}{\updefault}{\color[rgb]{0,0,0}$1/2$}%
}}}}
\put(3781,-1636){\makebox(0,0)[b]{\smash{{\SetFigFontNFSS{12}{14.4}{\rmdefault}{\mddefault}{\updefault}{\color[rgb]{0,0,0}$5/6$}%
}}}}
\put(3781,-2986){\makebox(0,0)[b]{\smash{{\SetFigFontNFSS{12}{14.4}{\rmdefault}{\mddefault}{\updefault}{\color[rgb]{0,0,0}$2/3$}%
}}}}
\put(4771,-2986){\makebox(0,0)[b]{\smash{{\SetFigFontNFSS{12}{14.4}{\rmdefault}{\mddefault}{\updefault}{\color[rgb]{0,0,0}$1/3$}%
}}}}
\put(6526,-2581){\makebox(0,0)[b]{\smash{{\SetFigFontNFSS{12}{14.4}{\rmdefault}{\mddefault}{\updefault}{\color[rgb]{0,0,0}$2/3$}%
}}}}
\put(6526,-1951){\makebox(0,0)[b]{\smash{{\SetFigFontNFSS{12}{14.4}{\rmdefault}{\mddefault}{\updefault}{\color[rgb]{0,0,0}$5/6$}%
}}}}
\put(6526,-3301){\makebox(0,0)[b]{\smash{{\SetFigFontNFSS{12}{14.4}{\rmdefault}{\mddefault}{\updefault}{\color[rgb]{0,0,0}$1/2$}%
}}}}
\put(5671,-3301){\makebox(0,0)[b]{\smash{{\SetFigFontNFSS{12}{14.4}{\rmdefault}{\mddefault}{\updefault}{\color[rgb]{0,0,0}$0$}%
}}}}
\put(5626,-2581){\makebox(0,0)[b]{\smash{{\SetFigFontNFSS{12}{14.4}{\rmdefault}{\mddefault}{\updefault}{\color[rgb]{0,0,0}$1/6$}%
}}}}
\put(5626,-1951){\makebox(0,0)[b]{\smash{{\SetFigFontNFSS{12}{14.4}{\rmdefault}{\mddefault}{\updefault}{\color[rgb]{0,0,0}$1/3$}%
}}}}
\put(4276,-3256){\makebox(0,0)[b]{\smash{{\SetFigFontNFSS{14}{16.8}{\rmdefault}{\mddefault}{\updefault}{\color[rgb]{0,0,0}$H_0$}%
}}}}
\put(4276,-1906){\makebox(0,0)[b]{\smash{{\SetFigFontNFSS{14}{16.8}{\rmdefault}{\mddefault}{\updefault}{\color[rgb]{0,0,0}$H_1$}%
}}}}
\put(6076,-2626){\makebox(0,0)[b]{\smash{{\SetFigFontNFSS{12}{14.4}{\rmdefault}{\mddefault}{\updefault}{\color[rgb]{0,0,0}$S_1$}%
}}}}
\put(6076,-3301){\makebox(0,0)[b]{\smash{{\SetFigFontNFSS{12}{14.4}{\rmdefault}{\mddefault}{\updefault}{\color[rgb]{0,0,0}$S_0$}%
}}}}
\put(6076,-1996){\makebox(0,0)[b]{\smash{{\SetFigFontNFSS{12}{14.4}{\rmdefault}{\mddefault}{\updefault}{\color[rgb]{0,0,0}$S_2$}%
}}}}
\end{picture}%

%% file: doubled_sfc.pdf_t.txt
\begin{picture}(0,0)%
\includegraphics{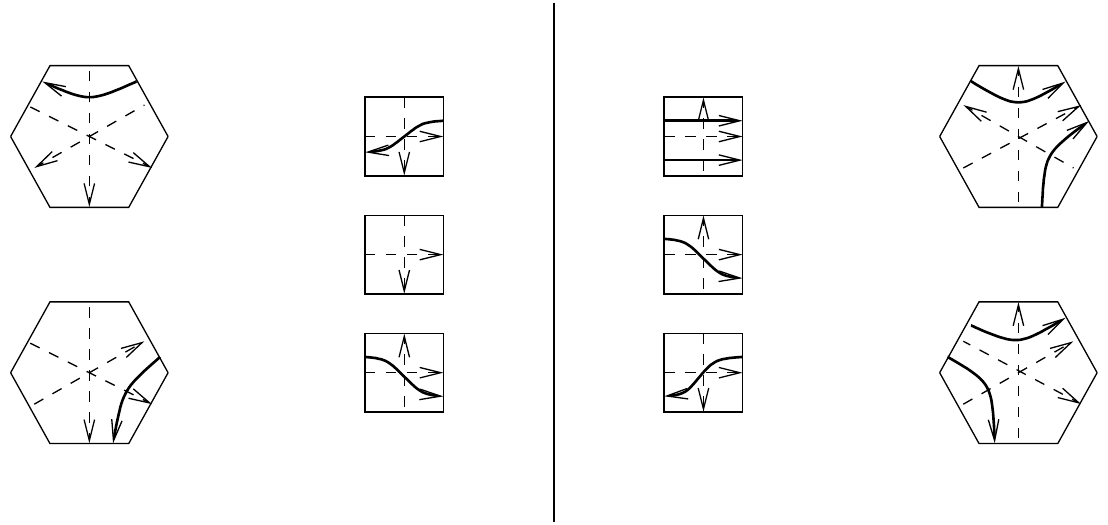}%
\end{picture}%
\setlength{\unitlength}{3315sp}%
\begingroup\makeatletter\ifx\SetFigFontNFSS\undefined%
\gdef\SetFigFontNFSS#1#2#3#4#5{%
  \reset@font\fontsize{#1}{#2pt}%
  \fontfamily{#3}\fontseries{#4}\fontshape{#5}%
  \selectfont}%
\fi\endgroup%
\begin{picture}(6330,2994)(1111,-3853)
\put(1621,-3616){\makebox(0,0)[b]{\smash{{\SetFigFontNFSS{10}{12.0}{\rmdefault}{\mddefault}{\updefault}{\color[rgb]{0,0,0}$0$}%
}}}}
\put(1621,-2266){\makebox(0,0)[b]{\smash{{\SetFigFontNFSS{10}{12.0}{\rmdefault}{\mddefault}{\updefault}{\color[rgb]{0,0,0}$1/6$}%
}}}}
\put(2116,-1411){\makebox(0,0)[b]{\smash{{\SetFigFontNFSS{10}{12.0}{\rmdefault}{\mddefault}{\updefault}{\color[rgb]{0,0,0}$1/2$}%
}}}}
\put(1126,-1411){\makebox(0,0)[b]{\smash{{\SetFigFontNFSS{10}{12.0}{\rmdefault}{\mddefault}{\updefault}{\color[rgb]{0,0,0}$5/6$}%
}}}}
\put(1126,-2761){\makebox(0,0)[b]{\smash{{\SetFigFontNFSS{10}{12.0}{\rmdefault}{\mddefault}{\updefault}{\color[rgb]{0,0,0}$2/3$}%
}}}}
\put(2116,-2761){\makebox(0,0)[b]{\smash{{\SetFigFontNFSS{10}{12.0}{\rmdefault}{\mddefault}{\updefault}{\color[rgb]{0,0,0}$1/3$}%
}}}}
\put(3871,-2356){\makebox(0,0)[b]{\smash{{\SetFigFontNFSS{10}{12.0}{\rmdefault}{\mddefault}{\updefault}{\color[rgb]{0,0,0}$2/3$}%
}}}}
\put(3871,-1726){\makebox(0,0)[b]{\smash{{\SetFigFontNFSS{10}{12.0}{\rmdefault}{\mddefault}{\updefault}{\color[rgb]{0,0,0}$5/6$}%
}}}}
\put(3871,-3076){\makebox(0,0)[b]{\smash{{\SetFigFontNFSS{10}{12.0}{\rmdefault}{\mddefault}{\updefault}{\color[rgb]{0,0,0}$1/2$}%
}}}}
\put(3016,-3076){\makebox(0,0)[b]{\smash{{\SetFigFontNFSS{10}{12.0}{\rmdefault}{\mddefault}{\updefault}{\color[rgb]{0,0,0}$0$}%
}}}}
\put(2971,-2356){\makebox(0,0)[b]{\smash{{\SetFigFontNFSS{10}{12.0}{\rmdefault}{\mddefault}{\updefault}{\color[rgb]{0,0,0}$1/6$}%
}}}}
\put(2971,-1726){\makebox(0,0)[b]{\smash{{\SetFigFontNFSS{10}{12.0}{\rmdefault}{\mddefault}{\updefault}{\color[rgb]{0,0,0}$1/3$}%
}}}}
\put(6931,-3616){\makebox(0,0)[b]{\smash{{\SetFigFontNFSS{10}{12.0}{\rmdefault}{\mddefault}{\updefault}{\color[rgb]{0,0,0}$0$}%
}}}}
\put(6931,-2266){\makebox(0,0)[b]{\smash{{\SetFigFontNFSS{10}{12.0}{\rmdefault}{\mddefault}{\updefault}{\color[rgb]{0,0,0}$1/6$}%
}}}}
\put(6436,-1411){\makebox(0,0)[b]{\smash{{\SetFigFontNFSS{10}{12.0}{\rmdefault}{\mddefault}{\updefault}{\color[rgb]{0,0,0}$1/2$}%
}}}}
\put(7426,-1411){\makebox(0,0)[b]{\smash{{\SetFigFontNFSS{10}{12.0}{\rmdefault}{\mddefault}{\updefault}{\color[rgb]{0,0,0}$5/6$}%
}}}}
\put(7426,-2761){\makebox(0,0)[b]{\smash{{\SetFigFontNFSS{10}{12.0}{\rmdefault}{\mddefault}{\updefault}{\color[rgb]{0,0,0}$2/3$}%
}}}}
\put(6436,-2761){\makebox(0,0)[b]{\smash{{\SetFigFontNFSS{10}{12.0}{\rmdefault}{\mddefault}{\updefault}{\color[rgb]{0,0,0}$1/3$}%
}}}}
\put(4681,-2356){\makebox(0,0)[b]{\smash{{\SetFigFontNFSS{10}{12.0}{\rmdefault}{\mddefault}{\updefault}{\color[rgb]{0,0,0}$2/3$}%
}}}}
\put(4681,-1726){\makebox(0,0)[b]{\smash{{\SetFigFontNFSS{10}{12.0}{\rmdefault}{\mddefault}{\updefault}{\color[rgb]{0,0,0}$5/6$}%
}}}}
\put(4681,-3076){\makebox(0,0)[b]{\smash{{\SetFigFontNFSS{10}{12.0}{\rmdefault}{\mddefault}{\updefault}{\color[rgb]{0,0,0}$1/2$}%
}}}}
\put(5536,-3076){\makebox(0,0)[b]{\smash{{\SetFigFontNFSS{10}{12.0}{\rmdefault}{\mddefault}{\updefault}{\color[rgb]{0,0,0}$0$}%
}}}}
\put(5581,-2356){\makebox(0,0)[b]{\smash{{\SetFigFontNFSS{10}{12.0}{\rmdefault}{\mddefault}{\updefault}{\color[rgb]{0,0,0}$1/6$}%
}}}}
\put(5581,-1726){\makebox(0,0)[b]{\smash{{\SetFigFontNFSS{10}{12.0}{\rmdefault}{\mddefault}{\updefault}{\color[rgb]{0,0,0}$1/3$}%
}}}}
\end{picture}%